\documentclass[12pt]{amsart}
\usepackage{a4wide}
\usepackage{amssymb}
\usepackage{amsthm}
\usepackage{amsmath}
\usepackage{amscd}
\usepackage{verbatim}
\usepackage{color}
\usepackage[mathscr]{eucal}
\textheight8.8in \textwidth6.55in \numberwithin{equation}{section}

\newcommand{\sm}[4]{\left(\begin{smallmatrix}#1&#2\\ #3&#4
\end{smallmatrix} \right)}
\newtheorem{theorem}{Theorem}
\newtheorem{lemma}[theorem]{Lemma}

\newtheorem{proposition}[theorem]{Proposition}

\theoremstyle{definition}
\newtheorem{definition}[theorem]{Definition}

\theoremstyle{remark}

\newtheorem{remark}{Remark}

\numberwithin{equation}{section}

\newcommand{\R}{\mathbb{R}}
\newcommand{\C}{\mathbb{C}}

\newcommand{\Z}{\mathbb{Z}}

\begin{document}
\title{Quasimodular forms and $s\ell(m|m)^\wedge$ characters}

\author{Kathrin Bringmann}
\address{Mathematical Institute,  University of Cologne, Weyertal 86-90, 50931 Cologne,
Germany} \email{kbringma@math.uni-koeln.de}
\author{Amanda Folsom}
\address{Yale University, Mathematics Department, P.O. Box 208283,
New Haven, CT 06520-8283} \email{amanda.folsom@yale.edu}
\author{Karl Mahlburg}
\address{Department of Mathematics, Louisiana State University, Baton Rouge, LA 70803} \email{mahlburg@math.lsu.edu}
\thanks{The research of the first author was supported by the Alfried Krupp Prize for Young University Teachers of the Krupp Foundation.  The second author is grateful for the support of National Science Foundation Grant DMS-1049553.  The third author was supported by NSF Grant DMS-1201435. }

\maketitle
\begin{abstract}
 In this paper, we establish automorphic properties and asymptotic behaviors of characters due to Kac-Wakimoto pertaining to $s\ell(m|n)^\wedge$ highest weight modules in the case $m=n$, extending work of the first author and Ono \cite{BOKac} and the first two authors \cite{BF} which pertains to the case $m>n$.
\end{abstract}

\section{Introduction}\label{introsec}
Modular forms and Lie algebras are intimately connected by the ``Monstrous Moonshine" phenomenon, which was first conjectured by Conway and Norton \cite{CN} in 1979, and was ultimately proved by Borcherds \cite{B} in 1991.  Moonshine relates the Fourier coefficients of the modular $j$-function
$$
j(\tau) = q^{-1} + 744 + 196884q + 21493760 q^2 + \cdots
$$
($q := e^{2 \pi i \tau}$), the Hauptmodul  for $\textnormal{SL}_2(\mathbb Z)$, to dimensions of irreducible representations of the Monster group, the largest of the sporadic finite simple groups.  An important predecessor of Moonshine was the seminal work of Kac \cite{K}, who first exhibited a link between infinite-dimensional Lie algebras and modular forms.  In particular, Kac's work in \cite{K} introduced the key idea that combinatorial identities can be obtained by comparing two different calculations of characters of representations of infinite-dimensional Lie algebras.

More recently, Kac and Wakimoto \cite{KW2} found a specialized character formula pertaining to the affine Lie superalgebra $s\ell(m,1)^\wedge$ for $tr_{L_{m,1}(\Lambda_{(r)})}q^{L_0}$ for any positive even integer $m$, where $L_{m,1}(\Lambda_{(r)})$ is the irreducible $s\ell(m,1)^\wedge$-module with highest weight $\Lambda_{(r)}$, and $L_0$ is the ``energy operator". In \cite{BOKac}, the first author and Ono answered a question of Kac regarding the ``modularity" of the characters $tr_{L_{m,1}(\Lambda_{(r))}}q^{L_0}$, and proved that in fact they are not modular forms, but are instead essentially holomorphic parts of harmonic weak Maass forms, which are non-holomorphic relatives to ordinary modular forms. In subsequent work \cite{F}, the second author related these characters to universal mock theta functions, and in \cite{BFasy}, the first two authors were able to exploit the ``mock-modularity" of the Kac-Wakimoto $s\ell(m,1)^\wedge$-module characters to exhibit their detail!
 ed asymptotic behaviors, extending results in \cite{KW2}.

Despite this progress made in understanding number theoretic properties of $s\ell(m,1)^\wedge$-characters, the more general problem of understanding  the ``modularity" and asymptotic behaviors of the Kac-Wakimoto $s\ell(m,n)^\wedge$-characters $tr_{L_{m,n}(\Lambda_{(r)})}q^{L_0}$ for arbitrary $n\in \mathbb N$ proved to be of a different nature than the case $n=1$.
Developing new methods, the first two authors \cite{BF} proved that in fact, the Kac-Wakimoto characters $tr_{L_{m,n}(\Lambda_{(r)})}q^{L_0}$ for integers $m>n$ are essentially the holomorphic parts of newly defined automorphic objects named ``almost harmonic Maass forms".
Loosely speaking these functions are sums of harmonic weak Maass forms under iterates of the raising operator multiplied by almost holomorphic modular forms.
The space of almost harmonic Maass forms contains harmonic weak Maass forms, which are the automorphic objects underlying the modern theory of mock modular forms \cite{BODys}, as a special class.  We refer the reader to \cite{BF} for more details on these automorphic functions.

We establish automorphic properties and asymptotic behaviors of characters due to Kac-Wakimoto pertaining to $s\ell(m|n)^\wedge$ highest weight modules in the case $m=n$, extending work of the first two authors \cite{BF} on the cases $m>n$.  For simplicity, we assume throughout that $m$ is even.

We point out that unlike the case of $n=1$ studied in \cite{BOKac}, one does not have the luxury of beginning with a multivariable Appell-Lerch sum expression for $tr_{L_{m,n}(\Lambda_{(r)})}q^{L_0}$  for arbitrary $n$.
Instead, we consider the generating function as given by Kac and Wakimoto for the specialized characters $\textnormal{ch}F_r(\tau)$
\begin{align}\label{KWgenf}
\textnormal{ch}F := \sum_{r \in \mathbb Z}  \textnormal{ch}F_r(\tau) \ \zeta^r
 =e^{\Lambda_0}\prod_{k\geq 1}
\frac{\prod_{i=1}^m\left(1+\zeta w_i q^{k-\frac{1}{2}}\right)\left(1+\zeta^{-1}w_i^{-1}q^{k-\frac{1}{2}}\right)}{\prod_{j=1}^{m}\left(1-\zeta
w_{m+j}
q^{k-\frac{1}{2}}\right)
\left(1-\zeta^{-1}w_{m+j}^{-1}q^{k-\frac{1}{2}}\right)},
\end{align}
 where $\textnormal{ch}F_r(\tau) = \textnormal{ch}F_{r,m}(\tau)$, $\zeta  := e^{2\pi i z}$, and $e^{\Lambda_0}$ and $w_s$ are certain operators \cite{KW2}, which are to be considered as formal variables.
To simplify the situation, we set $e^{\Lambda_0}$ and all $w_s$ equal to $1$ in \eqref{KWgenf}.  Once we have established the automorphic properties of the characters $\textnormal{ch}F_r(\tau)$, it is not difficult to deduce the automorphic properties of the characters $\textnormal{tr}_{L_{m,m}(\Lambda_{(r)})}q^{L_0}$ using the relationship (see \cite{KW2})
  \begin{align}\label{chtrgenrel}
  \textnormal{tr}_{L_{m, m}(\Lambda_{(r)})}q^{L_0} = \textnormal{ch}F_r(\tau) \cdot \prod_{k\geq 1}\left(1-q^k\right),
  \end{align}
  and the fact that $q^{\frac{1}{24}}\prod_{k\geq 1} (1-q^k) = \eta(\tau)$ is a well known modular form of weight $1/2$ (see \S \ref{autosec}).
We adopt the following notation for the specialized Kac-Wakimoto characters \eqref{KWgenf}:
\begin{equation}
\label{chphiformula1}
 \text{ch}F
 =i^m \varphi\left(z + \frac{\tau}{2};\tau\right),
\end{equation}
where $\varphi$ is an explicit quotient of Jacobi theta functions defined in \eqref{phidef}.

In light of \eqref{chphiformula1}, we proceed in our study of the Kac-Wakimoto characters by investigating the automorphic properties of the Fourier coefficients of $\varphi(z;\tau)$
  $$
  \varphi(z;\tau) = \sum_{r\in\mathbb Z} \chi_r(\tau) \zeta^r.
  $$
\begin{remark}\label{explain}  The automorphic properties of the Kac-Wakimoto characters $\textnormal{ch} F_r(\tau)$ may be deduced from those of $\chi_r(\tau)$, as discussed in \cite{BF}. More specifically, moving from $\varphi(z;\tau)$ to $\text{ch}F$ essentially corresponds to the elliptic shift $z \mapsto z +\frac{\tau}{2}$ applied to the function $\varphi(z;\tau)$ (see \S 3 of \cite{BF} for more detail). Note that the coefficients depend on the range in which $\zeta$ is taken.
Throughout we assume that $0<\text{Im}(z)<\text{Im}(\tau)$. We note that this (up to the boundary case) can be achieved by elliptic transformations.
By moving to a different range ``wallcrossing" occurs (see \cite{BF} for details).
\end{remark}

Our first main result establishes the automorphic properties of $\chi_r(\tau)$.  In what follows, we define the Nebentypus character $\psi$ for matrices $\gamma = \sm{a}{b}{c}{d} \in \Gamma_0(2)$ by
  \begin{align}\label{psichar} \psi\left(\gamma\right)=\psi_m(\gamma) := (-1)^{\frac{mc}{4}}.\end{align}

\noindent
Moreover, we require the well-known Eisenstein series $E_{2j}(\tau)$, which are defined in \eqref{eis}.  For $j\geq 2$, they  are holomorphic modular forms, while $E_2(\tau)$ is a \emph{quasimodular} form (see \S \ref{autosec}).  The Bernoulli numbers $B_{r}$
 are defined for integers $r\geq 0$ by the generating function
\begin{align}\label{bergen}
\frac{t}{e^t-1}=\sum_{r \geq 0} B_r \frac{t^r}{r!}.
\end{align}
 \begin{theorem}\label{prop1}
  For $r \in \Z$, we have
\begin{align*}
  \chi_r(\tau) &=  \frac{1}{1-q^{r}}\sum_{j=1}^{\frac{m}{2}} \frac{r^{2j-1}}{(2j-1)!}D_{2j}(\tau) ,  \  \ \ r\neq 0, \\
 \chi_0(\tau) &=
 D_0(\tau) +\sum_{j=1}^{\frac{m}{2}} \frac{ B_{2j}}{(2j)!} D_{2j}(\tau)E_{2j}(\tau),
\end{align*}
  where for each $0\leq j\leq \frac{m}{2}$, the function $D_{2j}(\tau) $ is a modular form of weight $-2j$ on $\Gamma_0(2)$ with Nebentypus character $\psi$, as defined in \eqref{psichar}.
 \end{theorem}

\begin{remark}
 The modular forms $D_{2j}(\tau)$, $0\leq j\leq m/2$, are explicitly computable.  We illustrate this fact by providing examples in \S \ref{examplesec} for the cases $m=2$ and $m=4$.
\end{remark}

Our methods follow those of \cite{DMZ} and the first two authors; however, we point out that the setting in which $n=m$ is of a different nature than the case $m>n$ as considered in \cite{BF}.  To this end, we establish the following decomposition of the form $\varphi(z;\tau)$ into a completed ``finite part" $\widehat{\varphi}^F(\tau)$ and a completed ``polar part" $\widehat{\varphi}^P(z;\tau)$, defined in \eqref{phifhat} and \eqref{phiphat}, respectively.  Moreover, we establish the automorphic properties of these parts.  (See \S \ref{22sec} for a discussion of almost holomorphic forms, and \S \ref{23sec} for a discussion of Jacobi forms.)

\begin{theorem}\label{prop2} We have that
$$
\varphi(z;\tau) = \widehat{\varphi}^F(\tau) + \widehat{\varphi}^P(z;\tau),
$$
where \begin{enumerate}
\item the function $\widehat{\varphi}^F(\tau)$ is an almost holomorphic modular form of weight $0$ on $\Gamma_0(2)$ with Nebentypus character $\psi$;
\item the function $\widehat{\varphi}^P(z;\tau)$ is a (non-holomorphic) Jacobi form of weight $0$ on $\Gamma_0(2)$ with Nebentypus character $\psi$.
     \end{enumerate}
\end{theorem}
We next consider the asymptotic behavior of the Kac-Wakimoto characters, generalizing results in
 \cite{BFasy}, \cite{BF},
 and \cite{KW2}.
\begin{theorem}\label{prop3}
 Let $r \in \mathbb Z$ and $m > 1 $.
If $\tau = it$, then for any
$N\in\mathbb N_0$, as $t\to 0^+$ we have
\[
{tr}_{L_{m, m}(\Lambda_{(r)})}q^{L_0}=e^{\frac{\pi t}{12}+\frac{\pi}{12 t}(3m-1)}\frac{\sqrt{t}}{2^m}
\left(\sum_{k=0}^Na_k(m,r)\frac{t^k}{ k!}+O\left(t^{N+1}\right)\right),
\]
where
\[
a_k(m,r):=(-2\pi i r)^k \mathscr E_{k,m}.
\]
Here the numbers $\mathscr E_{k,m}$ are ``higher" Euler numbers defined by the recurrence \eqref{recurEs} and the initial conditions   \eqref{E1defint} and \eqref{E2defint}.
\end{theorem}

The remainder of the paper is structured as follows.  In \S \ref{autosec}, we define the associated automorphic objects and describe their key properties. In \S \ref{proofsec}, we prove Theorem \ref{prop1}, Theorem \ref{prop2}, and Theorem \ref{prop3}.  In \S \ref{examplesec}, we provide explicit examples of Theorem \ref{prop1} for the cases $m=2$ and $m=4$.

\section{Automorphic forms}\label{autosec}
Here we provide definitions and properties of various automorphic objects relevant to our treatment of the Kac-Wakimoto characters $ \textnormal{tr}_{L_{m, m}(\Lambda_{(r)})}q^{L_0}$.

\subsection{Modular forms} Assume that $\kappa\in\frac12 \Z,$ and $\Gamma$ is a congruence subgroup of either $\text{SL}_2(\Z)$ or $\Gamma_0 (4),$ depending on whether or not $\kappa\in\Z$, respectively.
The weight $\kappa$ {\it slash operator}, defined for a matrix $\gamma = \left(\begin{smallmatrix}
a & b \\ c & d
\end{smallmatrix}\right) \in\Gamma$ and any function $f:\mathbb{H} \rightarrow \C,$  is given by
$$
f\big|_\kappa \gamma (\tau) := j(\gamma, \tau)^{-2\kappa} f\left( \frac{a\tau +b}{c\tau +d} \right)
$$
where the automorphy factor is
$$
j(\gamma, \tau) := \begin{cases}
   \sqrt{c\tau +d} & \text{if } \kappa\in\Z, \\
   \left( \frac{c}d \right)  \varepsilon_d^{-1} \sqrt{c\tau +d}    & \text{if } \kappa\in\frac12 \Z \backslash \Z,
  \end{cases}
$$
and
$$
\varepsilon_d:=\begin{cases}
   1 & \text{if } d\equiv 1 \pmod{4}, \\
   i       & \text{if } d\equiv 3 \pmod{4}.
  \end{cases}
$$

\begin{definition}
\label{D:modular}
Let $\kappa\in\frac12 \Z$, $N$ a positive integer, and $\chi$ a Dirichlet character modulo $N$. A holomorphic \it{modular form of weight $\kappa$ for $\Gamma$ with Nebentypus character $\chi$} is a holomorphic function $f:\mathbb{H}\rightarrow \C$ satisfying
\begin{enumerate}
\item For all $\gamma=\left(\begin{smallmatrix}
a & b \\ c & d
\end{smallmatrix}\right)\in\Gamma$ and all $\tau\in\mathbb{H}$, we have
$
f|_\kappa \gamma (\tau ) = \chi (d)f(\tau ). $
\item The function $f$ is holomorphic at all cusps of $\Gamma$.
\end{enumerate}
\end{definition}
\emph{Meromorphic modular forms} allow poles in $\mathbb{H}$ and the cusps and \emph{weakly holomorphic forms} only in the cusps.
Meromorphic modular forms and weakly holomorphic modular forms are described in detail in most standard textbooks on modular forms, including \cite{DS}.
A modular form required here is Dedekind's $\eta$-function, defined by
 \begin{equation}
 \eta(\tau) := q^\frac{1}{24}\prod_{k\geq 1} \left(1-q^k\right)\label{etadefn}.
\end{equation}
This function  is well known to satisfy the following transformation law \cite{Rad}.
\begin{lemma}\label{ETtrans}
 For
$\gamma=\sm{a}{b}{c}{d} \in \textnormal{SL}_2(\mathbb Z)$, we have that
\begin{equation*}
\eta\left(\gamma\tau\right)  = \rho\left(\gamma\right)(c\tau + d)^{\frac{1}{2}} \eta(\tau), \label{etatrgen}
\end{equation*}
where $\rho\left(\gamma\right)$ is a $24$th root of unity, which can be given explicitly in terms of Dedekind sums \cite{Rad}.     In particular, we have that
\begin{equation*}
\eta\left(-\frac{1}{\tau} \right)= \sqrt{-i \tau} \eta(\tau).
\end{equation*}
\end{lemma}
We also encounter the modular Eisenstein series $E_{2j}(\tau)$, defined for  integers $j\geq 1$ by\begin{align}\label{eis}
E_{2j}(\tau) := 1 - \frac{4j}{B_{2j}} \sum_{\ell\geq 1} \sigma_{2j-1}(\ell) q^{\ell},
\end{align}
where
$\sigma_r(\ell) := \sum_{d|\ell} d^r$, and $B_r$ denotes the $r$-th  Bernoulli number (cf. \eqref{bergen}).  For $j\geq 2$, it is well known that the Eisenstein series $E_{2j}(\tau)$ are modular forms of weight $2j$ on $\textnormal{SL}_2(\mathbb Z)$.

\subsection{Almost holomorphic modular forms and quasimodular forms}\label{22sec}

We also encounter \emph{almost holmorphic modular forms}, which, as originally defined by Kaneko-Zagier \cite{KZ}, are functions that transform like usual modular forms, but may also include non-holomorphic terms expressed as polynomials in $1/v$ with coefficients holomorphic in $\mathbb{H}$ allowing poles at the cusps, where throughout we write $\tau=u+iv$.

 Standard examples of almost holomorphic modular
form include derivatives of holomorphic modular forms, as well as the non-holomorphic Eisenstein series $\widehat{E}_2$, defined by
\begin{align}\label{e2hat}
\widehat{E}_2(\tau) := E_2(\tau) - \frac{3}{\pi v}.
\end{align}
In this case the ``holomorphic part" is simply
\begin{align}\label{E2holdef}E_2(\tau) = 1-24\sum_{n\geq 1} \sigma_1(n) q^n.
 \end{align}
In general, the holomorphic part of an almost holomorphic modular form is called a {\it quasimodular form}.  After their introduction \cite{KZ}, almost holomorphic modular forms have been shown to play numerous roles in mathematics and physics  (see for example
the work of Aganagic-Bouchard-Klemm \cite{ABK}).

 \subsection{Jacobi forms}\label{23sec}
Holomorphic Jacobi form are two-variable relatives to modular forms that were first defined by Eichler and Zagier \cite{EZ} as follows.
\begin{definition}\label{harmjacdef} A \emph{holomorphic Jacobi form of weight $\kappa$ and index $M$} ($\kappa, M \in \mathbb N$) on congruence a subgroup $\Gamma \subseteq \textnormal{SL}_2(\mathbb Z)$  is a holomorphic function
$\phi(z;\tau):\mathbb C \times \mathbb H \to
\mathbb C$ that satisfies the following properties for all $\gamma = \sm{a}{b}{c}{d} \in \Gamma$ and $\lambda,\mu \in \mathbb Z.$
 \begin{enumerate}\item
$\phi\!\left(\frac{z}{c\tau + d};\gamma\tau\right) = (c\tau + d)^\kappa e^{\frac{2\pi i M c z^2}{c\tau + d}} \phi(z;\tau)$,
\item $\phi(z + \lambda \tau + \mu;\tau) = e^{-2\pi i M (\lambda^2\tau + 2\lambda z)} \phi(z;\tau)$,
\item $\phi(z;\tau)$ has a Fourier development of the form $\sum_{n, r} c(n,r)q^n \zeta^r$, with $c(n,r)=0$ unless $n\geq r^2/4M$.
\end{enumerate}\end{definition}
\noindent For the sake of brevity, we have not given the most general definition above, but the theory of Jacobi forms developed in \cite{EZ} does also allow for half-integral weights, multiplier systems, and meromorphic functions, with the obvious modifications (cf. Definition \ref{D:modular}).

The most important Jacobi form used in our treatment is Jacobi's theta function, which is defined by
\begin{equation}
 \vartheta(z;\tau)=\vartheta(z):=\sum_{\nu\in\frac12+\Z}e^{\pi i \nu^2\tau+2\pi i\nu\left(z+\frac12\right)}.\label{thedefn}
\end{equation}
This function  is well known to satisfy the following transformation law \cite[(80.31) and (80.8)]{Rad}.
 \begin{lemma}\label{thetatrans}
   For $\lambda,\mu \in \mathbb Z$, and $\gamma=\sm{a}{b}{c}{d}\in \textnormal{SL}_2(\mathbb Z)$, we have that
 \begin{align}\label{thetae}
  \vartheta(z+\lambda \tau + \mu;\tau) &= (-1)^{\lambda + \mu} q^{-\frac{\lambda^2}{2}}e^{-2\pi i \lambda z} \vartheta(z;\tau), \\ \label{thetam}
   \vartheta\left(\frac{z}{c\tau+d};\gamma\tau\right) &= \rho^3(\gamma)  (c\tau + d)^{\frac{1}{2}}e^{\frac{\pi i c z^2}{c\tau + d}} \vartheta(z;\tau),
   \end{align}
    where $\rho(\gamma)$ is as defined in Lemma \ref{ETtrans}.
 \end{lemma}
\noindent We note that Lemma \ref{thetatrans} essentially says that $\vartheta$ is a Jacobi form of weight and index $1/2$ with multiplier $\rho^3$.

The Jacobi theta function also satisfies the well known triple product identity
 \begin{align}\label{JTP}
\vartheta(z;\tau) = -i q^{\frac{1}{8}} \zeta^{-\frac{1}{2}} \prod_{r\geq 1} (1-q^r)\left(1-\zeta q^{r-1}\right) \left(1-\zeta^{-1}q^r\right),
\end{align} where throughout $\zeta := e^{2\pi i z}$.

\section{Proof of the theorems}\label{proofsec}
As mentioned in \S \ref{introsec} (and contrary to the case $n=1$) we do not have a multivariable Appell-Lerch sum expression for $tr_{L_{m,n}(\Lambda_{(\ell)})}q^{L_0}$, and therefore we instead
consider the generating function given by Kac-Wakimoto for the specialized characters $\textnormal{ch}F_\ell$ as in \eqref{KWgenf}.
It is not difficult to see, using \eqref{JTP},  that we may rewrite the specialized Kac-Wakimoto generating function \eqref{KWgenf} as a quotient of Jacobi theta functions
\begin{equation} \label{chphiformula}
 \text{ch}F
 =i^m \varphi\left(z + \frac{\tau}{2};\tau\right),
\end{equation}
where  \begin{align}\label{phidef}
\varphi(z; \tau):=\left(\frac{\vartheta\left(z+\frac12;\tau \right)}{\vartheta(z;\tau)}\right)^m.
\end{align}
Given \eqref{chphiformula} and Remark \ref{explain} in \S \ref{introsec}, we proceed by investigating the automorphic properties of the Fourier coefficients  $\chi_r(\tau)$ of $\varphi(z;\tau)$ (as a function of $z$)
  $$
  \varphi(z;\tau) = \sum_{r\in\mathbb Z} \chi_r(\tau) \zeta^r.
  $$
 Using Lemma \ref{thetatrans}, noting that $m$ is even, we establish the following transformation properties of the function $\varphi(z;\tau)$.

 \begin{lemma} \label{phitrans}
For $\gamma=\sm{a}{b}{c}{d} \in \Gamma_0(2)$, and $\lambda,\mu \in \mathbb Z$, we have that \begin{align*} \varphi(z+\lambda \tau + \mu;\tau) &=\varphi(z;\tau), \\
\varphi\left(\frac{z}{c\tau+d};\gamma\tau\right) &= \psi(\gamma) \varphi(z;\tau).
\end{align*}
That is, $\varphi$ is a (meromorphic) Jacobi form of weight $0$ and index $0$ with Nebentypus character $\psi$.
\end{lemma}
Using the fact that $\vartheta(z;\tau)$ is an odd function in $z$ together with \eqref{thetae}, we find that  $\varphi(z;\tau)$ is an even function of $z$.  Furthermore, the function $\varphi(z;\tau)$ has a pole of order $m$ at $z=0$, so we may thus write the Laurent series expansion
 $$
 \varphi(z;\tau) = \frac{D_m(\tau)}{(2\pi i z)^m} + \frac{D_{m-2}(\tau)}{(2\pi i z)^{m-2}} + \cdots + \frac{D_2(\tau)}{(2\pi i z)^2} + D_0(\tau) + O\left(z^2\right).
 $$

Using Lemma \ref{phitrans}, it is not difficult to deduce the following modular transformation properties of the functions $D_{2j}$.
 \begin{lemma}\label{hjtrans}
 For all $\gamma = \sm{a}{b}{c}{d} \in \Gamma_0(2),$ we have that for each $0\leq j \leq m/2$
\begin{align*}
D_{2j}(\gamma\tau) &= \psi(\gamma)(c\tau+d)^{-2j} D_{2j}(\tau).
\end{align*}
Moreover the functions $D_{2j}$ are weakly holomorphic on $\mathbb{H}$.
\end{lemma}
Following \cite{DMZ}, we define the  ``finite part" $\varphi^F(\tau)$ and the ``polar part" $\varphi^P(z;\tau)$ of the form $\varphi(z;\tau)$ by
\begin{align}
\label{phif} \varphi^F(\tau) &:= D_0(\tau) +\sum_{j=1}^{\frac{m}{2}} \frac{ B_{2j}}{(2j)!} D_{2j}(\tau)E_{2j}(\tau), \\
\label{phip}
\varphi^P(z;\tau) &:=  \sum_{j=1}^\frac{m}{2} \frac{D_{2j}(\tau)}{(2j-1)!}\sum_{r\neq 0}\frac{r^{2j-1}}{1-q^r}\zeta^r,
\end{align}
where the Eisenstein series $E_{2j}(\tau)$ are defined in \eqref{eis}.
\begin{proposition}\label{prop5} We have that
\begin{align*}
 \varphi(z;\tau) &=   \varphi^F(\tau) + \varphi^P(z;\tau),
\end{align*}
where $\varphi^F(\tau)$ and $\varphi^P(z;\tau)$ are defined in \eqref{phif} and \eqref{phip}, respectively.
\end{proposition}
\begin{proof}
We start by rewriting \eqref{phip}. A direct calculation gives that
\[
\sum_{r\neq 0}\frac{r^{2j-1}}{1-q^r}\zeta^r=\partial_z^{2j-1}\left(\sum_{r\geq 0}\frac{1}{1-\zeta q^r}-\sum_{r>0}\frac{\zeta^{-1}q^r}{1-\zeta^{-1}q^r}\right)
\]
with $\partial_z:=\frac{1}{2\pi i}\frac{d}{dz}$.
 Next we note that we have
 $$
 \left(\sum_{r \geq 0} \frac{1}{1-\zeta q^r} -
 \sum_{r > 0} \frac{\zeta^{-1}q^r}{1-\zeta^{-1}q^r}\right)
=-\frac{1}{2\pi iz}+O(1),
$$
and thus
\begin{align}\label{diff1}
\partial_z^{2j-1} \left(\sum_{r \geq 0} \frac{1}{1-\zeta q^r}
- \sum_{r > 0} \frac{\zeta^{-1}q^r}{1-\zeta^{-1}q^r}\right)=\frac{(2j-1)!}{(2\pi iz)^{2j}}+O(1).
\end{align}
Thus, from   Lemma \ref{phitrans}, \eqref{phip}, and \eqref{diff1}, we deduce that $\varphi-\varphi^P$ is a bounded, holomorphic, function (in $z$), hence by Liouville's theorem, is constant.  To determine its value, we compute
\begin{equation}\label{compute}
\begin{split}
\varphi(z;\tau)-\varphi^P(z;\tau)&
=\lim_{z\to 0}\left(\varphi(z;\tau)-\sum_{j=1}^{\frac{m}{2}}\frac{D_{2j}(\tau)}{(2j-1)!}\partial_z^{2j-1}
\left(\frac{1}{1-\zeta}\right)\right)
\\&{\hspace{.8in}}
-\sum_{j=1}^{\frac{m}{2}}D_{2j}(\tau)\frac{\partial_z^{2j-1}}{(2j-1)!}
\left[\sum_{r>0}\frac{1}{1-\zeta q^r}-\sum_{r>0}\frac{\zeta^{-1}q^r}{1-\zeta^{-1}q^r}\right]_{z=0}.
\end{split}
\end{equation}
First, we find that
\begin{align}\nonumber
\partial_z^{2j-1}\left[\sum_{r>0}\frac{1}{1-\zeta q^r}
-\sum_{r>0}\frac{\zeta^{-1}q^r}{1-\zeta^{-1}q^r}\right]_{z=0}
&=\partial_z^{2j-1}\left[
\sum_{r>0\atop{\ell\geq 0}}\zeta^{\ell} q^{r \ell}
-\sum_{r>0\atop{\ell\geq 0}} \zeta^{-(\ell+1)}q^{r(\ell+1)}
\right]_{z=0}
\\ \label{part1}
&=2\sum_{r, \ell>0}\ell^{2j-1}q^{r \ell} = \frac{B_{2j}}{2j}\left(1-E_{2j}(\tau)\right).
\end{align}
Using the generating function for Bernoulli numbers \eqref{bergen}, we find moreover that
\begin{align}\label{part2}
\partial_z^{2j-1}\left[\frac{1}{1-\zeta}+\frac{1}{2\pi iz}\right]_{z=0}
=-\frac{B_{2j}}{2j}.
\end{align}
 Finally, we use the fact that
\begin{align}\label{part3}
\partial_z^{2j-1} \left(\frac{1}{2\pi i z}\right)
= -\frac{(2j-1)!}{(2 \pi i z)^{2j}}.
\end{align}
Inserting \eqref{part1}, \eqref{part2}, and \eqref{part3} into \eqref{compute},  we find that
   \begin{align*}
\varphi(z;\tau) - \varphi^P(z;\tau) =
D_0(\tau) + \sum_{j=1}^{\frac{m}{2}}\frac{B_{2j}}{(2j)!}D_{2j}(\tau)E_{2j}(\tau)=\varphi^F(\tau).
\end{align*}
\end{proof}
\begin{proof}[Proof of Theorem \ref{prop1}]  The expressions for $\chi_r$ in Theorem \ref{prop1} follow immediately in Proposition \ref{prop5}.
That the functions $D_{2j}(\tau)$ transform as claimed follows from Lemma \ref{hjtrans}.
\end{proof}
Next we define the following (non-holomorphic) completed functions
\begin{align}
\label{phifhat} \widehat{\varphi}^F(\tau) &:= \varphi^F(\tau) -\frac{D_2(\tau)}{4\pi v}, \\
\label{phiphat} \widehat{\varphi}^P(z;\tau) &:=  \varphi^P(z;\tau) + \frac{D_2(\tau)}{4\pi v}.
 \end{align}
\begin{proof}[Proof of Theorem \ref{prop2}]   Lemma \ref{hjtrans} implies that for $ 0 \leq j \leq \frac{m}{2}$, the function $D_{2j}(\tau)$ is a weight $-2j$ modular form on $\Gamma_0(2)$ with Nebentypus character $\psi$; we also know that the Eisenstein series $E_{2j}$ is modular of weight $2j$ on SL$_2(\Z)$ for $j > 1$.  Using these facts, as well as \eqref{e2hat}, we find that the completed function $\widehat{\varphi}^F$ is an almost holomorphic form of weight $0$ on $\Gamma_0(2)$ with Nebentypus character $\psi$, as claimed.   The remainder of Theorem \ref{prop2} follows by further appealing to Lemma \ref{phitrans}, Proposition \ref{prop5}, and the definition of $\widehat{\varphi}^P$ from \eqref{phiphat}.
\end{proof}
\begin{proof}[Proof of Theorem \ref{prop3}]
Since the proof is very similar to the case $m > n$, which is considered in detail in \cite{BF},
we only give a sketch of proof here.  We require certain generalized Euler numbers.
To be more precise, for integers $ k \geq 0$ we define:
 \begin{align}\label{E1defint}
\mathcal{E}_{ k, 1}&:= (-2i)^{-k} E_k,
\ \\ \label{E2defint}
 \mathcal{E}_{k, 2}&:=\frac{2i^{-k}}{\pi}B_k\left(\frac12\right),
\end{align}
where the Euler polynomials $E_k(z)$ and Bernoulli polynomials $B_k(z)$ are given by the generating functions
\begin{align*}
\sum_{k \geq 0} E_k(z) \frac{x^k}{k!} := \frac{2e^{xz}}{e^x + 1}, \qquad
\sum_{k \geq 0} B_k\left(z\right) \frac{x^k}{k!}:= \frac{xe^{zx}}{e^x - 1},
\end{align*}
and $E_k := E_k(0)$.
For any $n \geq 1$ and $k \geq 0$, we build on the initial functions \eqref{E1defint} and \eqref{E2defint} and define generalized Euler functions by the following recurrence:
\begin{align}\label{recurEs}
\mathcal{E}_{k, n+2}=\frac{n}{(n+1)}\mathcal{E}_{k, n}+\frac{k(k-1)}{\pi^2n(n+1)}\mathcal{E}_{k-2, n},
\end{align}
where we take $\mathcal{E}_{ k,n}:=0$ if $k<0$.

Using Cauchy's formula, it is easy to conclude from Lemma \ref{ETtrans} and Lemma \ref{thetatrans} that
\begin{align}\label{trform1}
 {tr}_{L_{m, m}(\Lambda_{(r)})}q^{L_0}
 =\frac{1}{{\sqrt{t}}2^m}   e^{ \frac{\pi t}{12} +\frac{\pi}{12t} (3m-1)  }
\int_{-\frac12}^{\frac12}\frac{e^{-2\pi i r u}}{\cosh\left(\frac{\pi u}{t}\right)^m}
\left(1+O\left(e^{\frac{2\pi |u|}{t}-\frac{\pi}{t}}\right)\right)du.
\end{align}
To complete the proof, it thus suffices to understand the asymptotic behavior of
\begin{align}\label{asympint1}
\int_{-\frac12}^{\frac12}\frac{1}{\cosh\left(\frac{\pi u}{t}\right)^m}
\left(e^{-2\pi i r u}+O\left(e^{\frac{2\pi |u|}{t}-\frac{\pi}{t}}\right)\right)du.
\end{align}
We extend the integral to all of $\R$, where a change of variables then gives
$$
t \int_{\R} \frac{e^{-2\pi i r u t}}{\cosh\left(\pi u\right)^m} du
= t \sum_{k=0}^{N} \frac{(-2 \pi ir t)^k}{k!} \int_{\R} \frac{u^k}{\cosh\left(\pi u\right)^m} du
+  O\left( t^{N+2}\right).
$$
These integrals were evaluted in \cite{BF}, where it was shown that
 $$
 \int_{\R} \frac{u^k}{\cosh\left(\pi u\right)^m} du = \mathcal{E}_{k,m}.
 $$

The error introduced by extending the integration range in \eqref{asympint1} may be bounded by
$$
\int_{\frac12}^{\infty} e^{- \frac{\pi m u}{t}} du \ll t e^{-\frac{\pi m }{2t}  }.
$$
Finally the error coming from the error term in \eqref{asympint1} may be bounded by
 $$
\int_{-\frac12}^{\frac12}\frac{ e^{\frac{2\pi |u|}{t}-\frac{\pi}{t}}  }{\cosh\left(\frac{\pi u}{t}\right)^m}
du
= 2 e^{-\frac{\pi}{t}} \int_{0}^{\frac12} \frac{ e^{\frac{2\pi u}{t}}  }{\cosh\left(\frac{\pi u}{t}\right)^m}
du \ll  e^{-\frac{\pi}{t}} \int_0^{\frac{1}{2}}  e^{\frac{\pi u}{t}(2 - m)} du \ll e^{-\frac{\pi}{t}}.
$$
\end{proof}

\section{Examples}\label{examplesec}
In this section we illustrate Theorem \ref{prop1} by providing examples of the first two cases. Our techniques follow the computations of \S 4 from \cite{BF}. \\

 \noindent {\bf{Example 1.  ($m=2$)} } We start by writing
 \begin{align*}
\vartheta\left(z+\frac12\right)&=\vartheta\left(\frac12\right)+\vartheta^{''}\left(\frac12\right)\frac{z^2}{2!}+O\left(z^4\right),\\
\vartheta^\ast(z) &=\vartheta^\ast(0)+\vartheta^{\ast ''}(0)\frac{z^2}{2!}+O\left(z^4\right),
 \end{align*}
 where $\vartheta^*(z) := \vartheta(z)/z$.
  We therefore have
  \begin{align}\label{m2expan}
\varphi(z)
=\frac{\left(\vartheta\left(\frac12\right)+\vartheta^{''}\left(\frac12\right)\frac{z^2}{2}+O\left(z^4\right)\right)^2}
{z^2\left(\vartheta^\ast(0)+\vartheta^{\ast ''}(0)\frac{z^2}{2}+O\Big(z^4 \Big)\right)^2} = \frac{1}{z^2}\frac{\vartheta\left(\frac12\right)^2}{\vartheta^\ast(0)^2}  \left(\!1\!+\!\left(\!\frac{\vartheta^{''}\left(\frac12\right)}{\vartheta\left(\frac12\right)}
-\frac{\vartheta^{\ast ''}(0)}{\vartheta^\ast(0)}\!\right)z^2\!+\! O\left(z^4\right)\!\right).
\end{align}  A direct calculation yields that
\begin{align}\label{thetase2} \begin{array}[7]{rclcrcl}
 \vartheta\left(\frac{1}{2};\tau\right)  &=&\displaystyle -2 \frac{\eta(2\tau)^2}{\eta(\tau)},  & &
 \vartheta^*(0;\tau) &= &-2\pi \eta^3(\tau), \\
\displaystyle \frac{1}{(2\pi i)^2}\frac{\vartheta^{''}\left(\frac12;\tau\right)}{\vartheta\left(\frac12;\tau\right)}
&=&\displaystyle -\frac{1}{12}E_2(\tau)+\frac13 E_2(2\tau), & & \displaystyle
\frac{1}{(2\pi i)^2}\frac{\vartheta^{\ast ''}(0;\tau)}{\vartheta^\ast(0;\tau)}&=&\frac{1}{12}E_2(\tau).
\end{array}\end{align}
Using these facts together with \eqref{m2expan}, we deduce  that
\begin{align*} D_2(\tau) &= -4\frac{\eta^4(2\tau)}{\eta^8(\tau)}, \ \ \textnormal{ and \ \ }
D_0(\tau) = \frac{2}{3}\frac{\eta^4(2\tau)}{\eta^8(\tau)}\left(E_2(\tau) - 2E_2(2\tau)\right).
\end{align*}

\ \\ \noindent {\bf{Example 2.  ($m=4$)} }
 We may write \begin{align*}
\vartheta\left(z+\frac12\right)&=\vartheta\left(\frac12\right)+\vartheta^{''}\left(\frac12\right)\frac{z^2}{2!}
+\vartheta^{(4)}\left(\frac12\right)\frac{z^4}{4!}+O\left(z^6\right),\\
\vartheta^\ast(z) &=\vartheta^\ast(0)+\vartheta^{\ast ''}(0)\frac{z^2}{2!}+\vartheta^{\ast (4)}(0)\frac{z^4}{4!}+O\left(z^6\right),
\end{align*}
so that
{\small{\begin{align}\nonumber
\varphi(z)
&=\frac{\left(\vartheta\left(\frac12\right)+\vartheta^{''}\left(\frac12\right)\frac{z^2}{2}+\vartheta^{(4)}\left(\frac12\right)\frac{z^4}{4!}+O\Big(z^6\Big)\right)^4}
{z^4\left(\vartheta^\ast(0)+\vartheta^{\ast ''}(0)\frac{z^2}{2}+\vartheta^{\ast (4)}(0)\frac{z^4}{4!}+O\Big(z^6 \Big)\right)^4}
\\ \nonumber
&=\frac{1}{z^4}\frac{\vartheta\left(\frac12\right)^4}{\vartheta^\ast(0)^4}  \Bigg(1+2\left(\frac{\vartheta^{''}\left(\frac12\right)}{\vartheta\left(\frac12\right)}
-\frac{\vartheta^{\ast ''}(0)}{\vartheta^\ast(0)}\right)z^2  +\frac{1}{6} \Bigg( 9 \left(\frac{\vartheta^{''}\!\!\left(\frac12\right)}{\vartheta\!\left(\frac12\right)}\right)^2 \!\!\!- 24 \left(\frac{\vartheta^{''}\!\!\left(\frac12\right)}{\vartheta\!\left(\frac12\right)}\right)\!\left(\frac{\vartheta^{* ''}\!\!\left(0\right)}{\vartheta^*\!\left(0\right)}\right) \\ &{\hspace{1.2in}} + 15 \left(\frac{\vartheta^{* ''}\!\!\left(0\right)}{\vartheta^*\!\left(0\right)}\right)^2 \!\!+ \left(\frac{\vartheta^{ (4)}\!\left(\frac12\right)}{\vartheta \!\left(\frac12\right)}\right) \!-\! \left(\frac{\vartheta^{* (4)}\!\left(0\right)}{\vartheta^* \!\left(0\right)}\right) \!\Bigg)z^4   + O\left(z^6\right)\Bigg).\label{quotlastlinecal}
\end{align}}}

In addition to \eqref{thetase2}, we similarly establish that the remaining  terms in \eqref{quotlastlinecal} satisfy
\begin{align}\label{theta4}
\frac{\vartheta^{(4)}\left(\frac12;\tau\right)}{\vartheta \left(\frac12;\tau\right)} &= (2\pi i)^4 \left(-\frac{7}{48} E_2^2(\tau) - \frac{1}{3} E_2^2(2\tau) + \frac{1}{24} E_4(\tau) + \frac{1}{2} E_2(\tau)E_2(2\tau)\right),\\ \label{thetastar4}
\frac{\vartheta^{* (4)}\left(0;\tau\right)}{\vartheta^*\left(0;\tau\right)} &= (2\pi i)^4\left(-\frac{1}{120}E_4(\tau) + \frac{1}{48}E_2^2(\tau)\right).
\end{align}
Using \eqref{thetase2}, \eqref{quotlastlinecal}, \eqref{theta4}, and \eqref{thetastar4},  we find that
\begin{align*}
D_4(\tau) &= 16\frac{\eta^8(2\tau)}{\eta^{16}(\tau)}, \\
D_2(\tau) &= -\frac{16}{3}\frac{\eta^8(2\tau)}{\eta^{16}(\tau)} \left(E_2(\tau) - 2 E_2(2\tau)\right), \\
D_0(\tau) &= \frac{4}{3} \frac{\eta^8(2\tau)}{\eta^{16}(\tau)}
\left(
\frac{1}{10}E_4(\tau) + \frac{1}{3}\left(E_2(\tau) - 2 E_2(2\tau)\right)^2
\right). \end{align*}

\end{document}